\IfFileExists{literature_discrete_rips_lemma.bib}{\immediate\write18{bibtex \jobname}}{}

\documentclass[a4paper,english,cleveref,nolineno]{socg-lipics-v2021}

\newenvironment{arxiv}{\ignorespaces}{}

\usepackage{amsthm,amssymb,amsmath}
\usepackage{tikz,tikz-cd}
\usetikzlibrary{decorations.pathmorphing}

\newcommand{\op}{\operatorname}
\newcommand{\CB}[2]{D_{#1}(#2)}
\newcommand{\VR}[2]{\op{Rips}_{#1}(#2)}
\newcommand{\diam}[1]{\op{diam}#1}
\newcommand{\Sph}[2]{S_{#1}(#2)}
\newcommand{\Vtx}[1]{\op{Vert}#1}
\newcommand{\cl}[1]{\op{Cl}#1}
\newcommand{\dGH}[2]{d_{\op{GH}}(#1,#2)}
\newcommand{\dis}{\op{dis}}
\newcommand{\hyp}[1]{\delta(#1)}
\newcommand{\K}[1]{K_{#1}}
\newcommand{\RC}[1]{C_{#1}}
\newcommand{\St}[1]{\operatorname{St}#1}
\newcommand{\card}[1]{\operatorname{card}#1}
\DeclareMathOperator{\cech}{\check{C}ech}
\DeclareMathOperator{\pr}{pr}

\makeatletter
\providecommand{\leftsquigarrow}{%
  \mathrel{\mathpalette\reflect@squig\relax}%
}
\newcommand{\reflect@squig}[2]{%
  \reflectbox{$\m@th#1\rightsquigarrow$}%
}
\makeatother

\title{Gromov Hyperbolicity, Geodesic Defect, and Apparent Pairs in Vietoris--Rips Filtrations}
\titlerunning{Gromov Hyperbolicity, Geodesic Defect, and Apparent Pairs in Rips Filtrations}
\author{Ulrich Bauer}{Department of Mathematics and Munich Data Science Institute, Technical University of Munich (TUM), Germany \and 
\url{www.ulrich-bauer.org}}{mail@ulrich-bauer.org}{https://orcid.org/0000-0002-9683-0724}{}%
\author{Fabian Roll}{Department of Mathematics, Technical University of Munich (TUM), Germany \and \url{https://www.roll.science}}{fabian.roll@tum.de}{https://orcid.org/0000-0002-3604-4545}{}%
\authorrunning{U. Bauer and F. Roll} %
\Copyright{Ulrich Bauer and Fabian Roll} %

\category{} %
\funding{This research has been supported by the DFG Collaborative Research Center SFB/TRR 109 \emph{Discretization in Geometry and Dynamics}.}%
\acknowledgements{We thank Michael Bleher, Lukas Hahn, and Andreas Ott for stimulating discussions about applications of Vietoris--Rips complexes to the topological study of coronavirus evolution motivating our interest in the persistence of tree metrics,
the organizers and participants of the AATRN Vietoris--Rips online seminar for sparking our interest in the Contractibility Lemma,
and the anonymous reviewers for valuable feedback.}

\hideLIPIcs

\begin{document}
\maketitle
\begin{abstract}
Motivated by computational aspects of persistent homology for Vietoris–Rips filtrations, we generalize a result of Eliyahu Rips on the contractibility of Vietoris–Rips complexes of geodesic spaces for a suitable parameter depending on the hyperbolicity of the space. We consider the notion of geodesic defect to extend this result to general metric spaces in a way that is also compatible with the filtration. We further show that for finite tree metrics the Vietoris–Rips complexes collapse to their corresponding subforests. We relate our result to modern computational methods by showing that these collapses are induced by the apparent pairs gradient, which is used as an algorithmic optimization in Ripser, explaining its particularly strong performance on tree-like metric data.
\end{abstract}

\section{Introduction}

The \emph{Vietoris--Rips} complex is a fundamental construction in algebraic, geometric, and applied topology.
For a metric space $X$ and a threshold $t > 0$, it is defined as the simplicial complex consisting of nonempty and finite subsets of $X$ with diameter at most $t$:
\[\VR{t}{X}=\{\emptyset\neq S\subseteq X\mid S\text{ finite, }\diam{S\leq} t\}.\]
First introduced by Vietoris~\cite{MR1512371} in order to make homology applicable to general compact metric spaces, it has also found important applications in geometric group theory \cite{Gromov1987} and topological data analysis \cite{SPBG:SPBG04:157-166}.
The role of the threshold in these three application areas is notably different.
The homology theory defined by Vietoris arises in the limit $t \to 0$.
In contrast, the key applications in geometric group theory rely on the fact that the Vietoris--Rips complex of a hyperbolic geodesic space is contractible for a sufficiently large threshold.
This observation, originally due to Rips and first published in Gromov's seminal paper on hyperbolic groups~\cite{Gromov1987}, is a fundamental result about the topology of Vietoris--Rips complexes and plays a central role in the theory of hyperbolic groups.
\begin{lemma}[Contractibility Lemma; Rips, Gromov~\cite{Gromov1987}]
  Let $X$ be a $\delta$-hyperbolic geodesic metric space. Then the complex $\VR{t}{X}$ is contractible for every $t > 0$ with $t \geq 4\delta$.
\end{lemma}
Here, a metric space $(X,d)$ is called \emph{geodesic} if for any two points $x,y\in X$ there exists an isometric map $[0,d(x,y)]\to X$ such that $0\mapsto x$ and $d(x,y)\mapsto y$, and it is called \emph{$\delta$-hyperbolic} (in the sense of Gromov \cite{Gromov1987}) for $\delta\geq 0$ if for any four points $w,x,y,z\in X$ we have
\begin{equation}
\label{gromov-hyperbolic-alternative}
 d(w,x)+d(y,z)\leq \max\{d(w,y)+d(x,z),d(w,z)+d(x,y) \}+2\delta.
\end{equation}
Finally, in applications of Vietoris--Rips complexes to topological data analysis, one is typically interested in the \emph{persistent homology} of the entire filtration of complexes for all possible thresholds.
A notable difference to the classical applications is that the metric spaces under consideration are typically finite, and in particular not geodesic.
This motivates the interest in a meaningful generalization of the Contractibility Lemma to finite metric spaces.
Based on the notion of a \emph{discretely geodesic space} defined by Lang \cite{lang}, which is the natural setting for hyperbolic groups,
and motivated by techniques used in that paper, we consider the following quantitative geometric property (called \emph{$\nu$-almost geodesic} in \cite[p. 271]{MR2883376}).
\begin{definition}
 A metric space $X$ is \emph{$\nu$-geodesic} if for all $x,y\in X$ and $r,s\geq 0$ with $r+s=d(x,y)$ there exists a point $z\in X$ with $d(x,z)\leq r+\nu$ and $d(y,z)\leq s+\nu$.
The \emph{geodesic defect} of $X$, denoted by $\nu(X)$, is the infimum over all $\nu$ such that $X$ is $\nu$-geodesic.
\end{definition}

Our first main result is a generalization of the Contractibility Lemma that also applies to non-geodesic spaces using the notion of geodesic defect, and further produces collapses that are compatible with the Vietoris--Rips filtration above the collapsibility threshold.

\begin{theorem}
\label{discrete-rips-fitlered-collapse}
Let $X$ be a finite $\delta$-hyperbolic $\nu$-geodesic metric space.
Then there exists a discrete gradient
that induces, for every $u > t\geq 4\delta+2\nu$, a sequence of collapses
 \[\VR{u}{X}\searrow \VR{t}{X} \searrow \{*\}
 .\]
\end{theorem}

\begin{example}
\label{geod-defect-tree}
An important special case is given by a finite \emph{tree metric space} $(V,d)$, where~$V$ is the vertex set of a positively weighted tree $T=(V,E)$, and where the edge weights are taken as lengths and $d$ is the associated path length metric, i.e., for two points $x,y\in V$ their distance is the infimum total weight of any path starting in $x$ and ending in~$y$.
The geodesic defect is~$\nu(V)=\frac12 \max_{e\in E}{l}(e)$, where $l(e)$ is the length of the edge~$e$.
Moreover, $(V,d)$ is $0$-hyperbolic (see 
\cite[Theorems 3.38 and 3.40]{MR2351587} 
for a characterization of $0$-hyperbolic spaces)%
.%
\end{example}
This example is of particular relevance in the context of evolutionary biology, where persistent Vietoris--Rips homology has been successfully applied to identify recombinations and recurrent mutations \cite{chan2013topology,MR4071389,BleherHahnEtAl2021}.
The metrics arising as genetic distances of aligned RNA or DNA sequences are typically very similar to trees, capturing the phylogeny of the evolution.
This motivates our interest in the particular case of tree metrics. These metric spaces are known to have acyclic Vietoris--Rips homology in degree $>0$, and so any homology is an indication of some evolutionarily relevant phenomenon.

Our second main result is a strengthened version of \cref{discrete-rips-fitlered-collapse} for the special case of tree metric spaces that connects the collapses of the Vietoris--Rips filtration to the construction of \emph{apparent pairs}, which play an important role as a computational shortcut in the software Ripser \cite{MR4298669}.
This result depends on a particular ordering of the vertices: we say that a total order of $V$ is \emph{compatible} with the tree $T$ if it extends the unique tree partial order resulting from choosing some arbitrary root vertex as the minimal element.

\begin{theorem}
\label{discrete-rips-tree-fitlered-collapse}
Let $V$ be a finite tree metric space for a weighted tree $T=(V,E)$, whose vertices are totally ordered in a compatible way.
Then the apparent pairs gradient for the lexicographically refined Vietoris--Rips filtration
induces a sequence of collapses
 \[\VR{u}{V}\searrow \VR{t}{V} \searrow T_t\]
for every $u > t > 0$ such that no edge $e \in E$ has length $l(e) \in (t,u]$,
where $T_t$ is the subforest with vertices $V$ and all edges of $E$ with length at most $t$.
In particular, the persistent homology of the Vietoris--Rips filtration is trivial in degree $>0$.
\end{theorem}

In the special case of trees with unit edge length, the proofs in \cite[Proposition 2.2]{MR3096593} and \cite[Proposition 3]{MR4130978} are similar in spirit to our proof of \cref{tree-collapse}, which is based on discrete Morse theory. Related results about implications of the geometry of a metric space on the homotopy types of the associated Vietoris--Rips complexes can be found in \cite{MR3003901,MR3968598,MR1879057}.

 \begin{remark}
  Given a vertex order $\leq$, the lexicographic order on simplices for the reverse vertex order~$\geq$ coincides with the reverse colexicographic order for the original order~$\leq$, which is used for computations in Ripser. %
  As a consequence, when the input is a tree metric with the points ordered in reverse order of the distances to some arbitrarily chosen root, then Ripser will identify all non-tree simplices in apparent pairs, requiring not a single column operation to compute its trivial persistent homology.
  In practice, we observe that on data that is almost tree-like, such as genetic evolution distances, Ripser exhibits exceptionally good computational performance.
  The results of this paper provide a partial geometric explanation for this behavior and yield a heuristic for preprocessing tree-like data by sorting the points to speed up the computation in such cases.
  In the application to the study of SARS-CoV-2 described in \cite{BleherHahnEtAl2021}, ordering the genome sequences in reverse chronological order, as an approximation of the reverse tree order for the phylogenetic tree, lead to a huge performance improvement, bringing down the computation time for the persistence barcode from a full day to about 2 minutes.
 \end{remark}

\section{Preliminaries}

\subsection{Discrete Morse theory and the apparent pairs gradient}

A \emph{simplicial complex $K$} on a {vertex set} $\Vtx K$ is a collection of nonempty finite subsets of $\Vtx K$ such that for any set $\sigma\in K$ and any nonempty subset $\rho\subseteq\sigma$ one has $\rho\in K$.
A set~$\sigma \in K$ is called a \emph{simplex}, and $\dim\sigma=\card\sigma-1$ is its \emph{dimension}.
Moreover, $\rho$ is said to be a \emph{face of $\sigma$} and $\sigma$ a \emph{coface of $\rho$}.
If $\dim\rho=\dim\sigma-1$, then we call $\rho$ a \emph{facet of $\sigma$} and $\sigma$ a \emph{cofacet of $\rho$}.
The \emph{star of $\sigma$}, $\St{\sigma}$, is the set of cofaces of $\sigma$ in $K$, and the \emph{closure of $\sigma$}, $\cl{\sigma}$, is the set of its faces. For a subset $E\subseteq K$, we write $\St{E}=\bigcup_{e\in E}\St{e}$.

Generalizing the ideas of Forman \cite{MR1612391}, a function $f\colon K\to \mathbb{R}$ is a \emph{discrete Morse function} \cite{MR3605986,MR2537376} if 
$f$ is monotonic, i.e., for any $\sigma,\tau \in K$ with $\sigma \subseteq \tau$ we have $f(\sigma)\leq f(\tau)$, and
there exists a partition of $K$ into intervals $[\rho,\phi] = \{ \psi \in K \mid \rho \subseteq \psi \subseteq \phi \}$ in the face poset such that $f(\sigma) = f(\tau)$ for any $\sigma \subseteq \tau$ if and only if $\sigma$ and $\tau$ belong to a common interval in the partition.
The collection of \emph{regular} intervals, $[\rho,\phi]$ with $\rho \neq \phi$, is called the \emph{discrete gradient of $f$}, and any singleton interval $[\sigma,\sigma]$, as well as the corresponding simplex $\sigma$, is called \emph{critical}.

\begin{proposition}[Hersh {\cite[Lemma 4.1]{MR2164921}}; Jonsson {\cite[Lemma 4.2]{MR2368284}}]
\label{union-of-gradients}
Let $K$ be a finite simplicial complex, and $\{K_\alpha\}_{\alpha \in A}$ a set of subcomplexes covering $K$, each equipped with a discrete gradient $V_\alpha$, such that for any simplex of $K$
\begin{itemize}
\item
there is a unique minimal subcomplex $K_\alpha$ containing that simplex, and
\item
the simplex is critical for the discrete gradients of all other such subcomplexes.
\end{itemize}
Then the regular intervals in the $V_\alpha$ are disjoint, and their union is a discrete gradient on~$K$.
\end{proposition}

An \emph{elementary collapse} $K\searrow K\setminus \{\sigma,\tau\}$ is the removal of a pair of simplices, where $\sigma$ is a facet of $\tau$, with $\tau$ the unique proper coface of $\sigma$. A \emph{collapse} $K\searrow L$ onto a subcomplex $L$ is a sequence of elementary collapses starting in $K$ and ending in $L$. An elementary collapse can be realized continuously by a strong deformation retraction and therefore collapses preserve the homotopy type. A discrete gradient can encode a collapse.
\begin{proposition}[{Forman \cite{MR1612391}; see also \cite[Theorem 10.9]{MR4249617}}]
\label{collapsing-theorem}
 Let $K$ be a finite simplicial complex and let $L\subseteq K$ be a subcomplex. Assume that $V$ is a discrete gradient on $K$ such that the complement $K\setminus L$ is the union of intervals in $V$. Then there exists a collapse $K\searrow L$.
\end{proposition}

 Let $f\colon K\to \mathbb{R}$ be a monotonic function. Assume that the vertices of $K$ are totally ordered. The \emph{$f$-lexicographic order} is the total order $\leq_f$ on $K$ given by ordering the simplices
 \begin{itemize}
  \item by their value under $f$,
  \item then by dimension,
  \item then by the lexicographic order induced by the total vertex order.
 \end{itemize}

We call a pair $(\sigma,\tau)$ of simplices in $K$ a \emph{zero persistence pair} if $f(\sigma)=f(\tau)$. An \emph{apparent pair} $(\sigma,\tau)$ with respect to the $f$-lexicographic order is a pair of simplices in~$K$ such that $\sigma$ is the maximal facet of $\tau$, and $\tau$ is the minimal cofacet of $\sigma$. The collection of apparent pairs forms a discrete gradient \cite[Lemma 3.5]{MR4298669}, called the \emph{apparent pairs gradient}.

Assume that $K$ is finite and $f\colon K\to \mathbb{R}$ a discrete Morse function with discrete gradient~$V$. Refine $V$ to another discrete gradient
\[\widetilde V=\{(\psi\setminus \{v\},\psi\cup\{v\})\mid \psi\in  [\rho,\phi]\in V,\ v=\min(\phi\setminus\rho) \}\]
by doing a minimal vertex refinement on each interval.\begin{lemma}
	\label{apparent-equal-gradient}
	 The zero persistence apparent pairs with respect to the $f$-lexicographic order are precisely the gradient pairs of $\widetilde V$.
\end{lemma}
\begin{proof}
	Let $(\sigma,\tau)$ be a zero persistence apparent pair. Then $f(\sigma)=f(\tau)$, and $\sigma$ and $\tau$ are contained in the same regular interval $I=[\rho,\phi]$ of $V$. Let $v$ be the minimal vertex in $\phi\setminus \rho$. By assumption,
	$\sigma$ is the maximal facet of $\tau$, and $\tau$ is the minimal cofacet of $\sigma$. Hence, $\sigma$ is lexicographically maximal among all facets of $\tau$ in~$I$, and $\tau$ is lexicographically minimal under all cofacets of $\sigma$ in~$I$.
	By the assumption that $(\sigma,\tau)$ forms an apparent pair, we cannot have $v \in \sigma$,
	as otherwise $\tau \setminus \{v\}$ would be a larger facet of $\tau$ than $\sigma$.
	Similarly, we cannot have $v \notin \tau$,
	as otherwise $\sigma \cup \{v\}$ would be a smaller cofacet of $\sigma$ than $\tau$.
	This means that~$\tau=\sigma\cup \{v \}$ and therefore $\{\sigma,\tau\}\in \widetilde V$.
	
	Conversely, assume that $\{\sigma,\tau\}\in\widetilde V$ holds. Consider the interval $I=[\rho,\phi]$ of $V$ with~$\{\sigma,\tau\}\subseteq I$ and let $v$ be the minimal vertex in $\phi\setminus \rho$. By construction of $\widetilde{V}$, $\sigma=\tau\setminus\{v\}$ is the lexicographically maximal facet of $\tau$ in $I$ and $\tau=\sigma\cup\{v\}$ is the lexicographically minimal cofacet of $\sigma$ in $I$. Therefore, $(\sigma,\tau)$ is a zero persistence apparent pair.
\end{proof}

\begin{arxiv}
\subsection{Hyperbolicity and geodesic defect of metric spaces}
In this section, we establish some basic facts about the hyperbolicity and the geodesic defect of metric spaces, such as their stability with respect to the Gromov–Hausdorff distance.
While this is already known for the hyperbolicity \cite{MemoliOkutanEtAl2021}, the stability of the geodesic defect, to the best of our knowledge, was not explicitly addressed before.
The results shown here are independent of the results in the following sections.

If $X$ is $\delta$-hyperbolic, then it is also $\delta'$-hyperbolic for every $\delta'\geq \delta$.
With this in mind, it is natural to consider the infimum over all $\delta$ such that $X$ is~$\delta$-hyperbolic, which is called the \emph{hyperbolicity} $\hyp{X}$ of $X$.
It follows from \cref{gromov-hyperbolic-alternative} that we have
\begin{equation}
\hyp{X} =  \sup_{w,x,y,z\in X}\frac{
d(w,x)+d(y,z) - \max\{d(w,y)+d(x,z),d(w,z)+d(x,y)\}}{2},
\end{equation}
and from this equivalent description of the hyperbolicity it can be seen that $X$ is indeed $\hyp{X}$-hyperbolic.
Moreover, note that every subspace of a $\delta$-hyperbolic space is also $\delta$-hyperbolic.
We now demonstrate that similar spaces have similar hyperbolicity.

A \emph{correspondence}~$C$ between two metric spaces $X$ and $Y$ is a subset $C\subseteq X\times Y$ such that $\pr_X(C)=X$ and $\pr_Y(C)=Y$. The \emph{Gromov--Hausdorff distance} between $X$ and $Y$ is defined as
\[
\dGH{X}{Y}=\frac{1}{2}\inf\{\dis(C)\mid C \text{ correspondence between $X$ and $Y$}\},
\] where $\dis(C)=\sup\{|d_X(x,x')-d_Y(y,y') |\mid (x,y),(x',y')\in C \}$ is the distortion of $C$.

\begin{proposition}[Mémoli et al.~\cite{MemoliOkutanEtAl2021}]
\label{hyperbolic-stability}
Let $X$ and $Y$ be metric spaces and $s=\dGH{X}{Y}$. If $X$ is $\delta$-hyperbolic, then $Y$ is $(\delta+4s)$-hyperbolic.
Hence, $|\hyp{X}-\hyp{Y}| \leq 4\dGH{X}{Y}$.
\end{proposition}

We now turn to the geodesic defect.
We start by establishing lower and upper bounds on the geodesic defect, relating it to other familiar geometric quantities.

\begin{proposition}
\label{gd-geq-zero}
 For any metric space $X$ we have $\nu(X)\geq \frac12\inf_{x\neq y}d(x,y)$.
\end{proposition}
\begin{proof}
 If $X$ is not $\nu$-geodesic for any $\nu$, then $\nu(X)=\infty$ and there is nothing to prove. Thus, assume that $X$ is $\nu$-geodesic. Let $\epsilon>0$ be arbitrary and let $u,w\in X$ be any two points with~$u\neq w$ and $d(u,w)-\epsilon\leq I:=\inf_{x\neq y}d(x,y)$. Then any other point has distance at least $d(u,w)-\epsilon$ to $u$ and $w$. As $X$ is $\nu$-geodesic, there exists a point $z\in X$ with $d(u,z)\leq \frac12 d(u,w)+\nu$ and $d(w,z)\leq \frac12 d(u,w)+\nu$. If $z=w$, then the first inequality implies $\frac12 I\leq \frac12 d(u,w)\leq \nu$ and hence $\frac12 I\leq \nu(X)$, because $\nu(X)$ is the infimum over all $\nu$ such that~$X$ is $\nu$-geodesic. If $z\neq w$, then
 \[d(u,w)-\epsilon\leq I\leq d(w,z)\leq \frac12 d(u,w)+\nu\]
 and therefore $\frac12 I-\epsilon \leq \frac12 d(u,w)-\epsilon\leq \nu$. Letting $\epsilon$ tend to zero implies $\frac12 I\leq \nu$ and hence~$\frac12 I\leq \nu(X)$, because $\nu(X)$ is the infimum over all $\nu$ such that $X$ is $\nu$-geodesic.
\end{proof}

A metric subset $X\subseteq Y$ is \emph{$r$-dense} if for every $y\in Y$ there exists $x\in X$ with $d(y,x)\leq r$. The following proposition proves an upper bound on the geodesic defect for $r$-dense subsets of a geodesic space. A partial converse for $\delta$-hyperbolic spaces is given by \cref{dense-embedding-injective-hull}.

\begin{proposition}
\label{subset-geodesic-small-defect}
 Let $X$ be an $r$-dense subset of a geodesic metric space $Y$. Then $X$ is~$r$-geodesic. In particular, $\nu(X)\leq r$.
\end{proposition}
\begin{proof}
  Let $x,y\in X$ be any two points and let $t,s\geq 0$ be with $t+s=d(x,y)$. Choose an isometric embedding $\gamma\colon [0,d(x,y)]\to Y$ with $\gamma(0)= x$ and $\gamma(d(x,y))= y$. As $X$ is $r$-dense, there exists a point $z\in X$ with $d(\gamma(t),z)\leq r$. By the triangle inequality, we get $d(x,z)\leq d(\gamma(0),\gamma(t))+d(\gamma(t),z)\leq t+r$ and similarly $d(y,z)\leq s+r$. This proves that $X$ is $r$-geodesic.
\end{proof}

A $0$-geodesic space is also called \emph{metrically convex} \cite{MR1904284}. If $X$ is a geodesic space, then its geodesic defect is zero, but the converse is not always true. In fact, for any \emph{length space}~%
\cite{MR1835418}, i.e., a metric space where the distance between two points is the infimum of lengths of paths connecting them, the geodesic defect is zero. It is worth noting that a complete metric space is a length space if and only if its geodesic defect is zero, and it is a geodesic space if and only if it is $0$-geodesic %
\cite[Section 2.4]{MR1835418}. The punctured unit disk in $\mathbb R^2$ is an example for a space that has geodesic defect~$0$ but that is not $0$-geodesic.
However, we have the~following.
\begin{lemma}
 Let $X$ be a proper metric space, i.e., assume that every closed metric ball is compact. Then~$X$ is~$\nu(X)$-geodesic.
\end{lemma}
\begin{proof}
 Let $x,y\in X$ be two points and $r,s\geq 0$ with $r+s=d(x,y)$. For every natural number $n\in \mathbb{N}$ the space $X$ is $\nu_n$-geodesic, where $\nu_n=\nu(X)+\frac1n$, and hence there exists a point $z_n\in X$ with $d(x,z_n)\leq r+\nu_n$ and $d(y,z_n)\leq s+\nu_n$. This sequence is contained in the closed ball of radius $r+\nu(X)+1$ centered at $x$, which is compact by assumption. Hence, there exists a convergent subsequence $z_{n_k}\to z$. The limit point $z\in X$ satisfies~$d(x,z)\leq r+\nu(X)$ and $d(y,z)\leq s+\nu(X)$. Therefore, $X$ is~$\nu(X)$-geodesic.
\end{proof}

We now show that the geodesic defect, like the hyperbolicity, is a Gromov--Hausdorff stable quantity.
\begin{proposition}
\label{geode-defect-stability}
Let $X$ and $Y$ be metric spaces and $s>\dGH{X}{Y}$.
If $X$ is $\nu$-geodesic, then $Y$ is $(\nu+3s)$-geodesic.
Hence, $|\nu(X)-\nu(Y)| \leq 3\dGH{X}{Y}$.
\end{proposition}
\begin{proof}
Let $C$ be a correspondence between $X$ and $Y$ with $\dis C \leq 2s$. Furthermore, let~$y,y'\in Y$ be any two points, and let $r,t\geq 0$ be such that $r+t=d(y,y')$. Choose two corresponding points $x,x'\in X$, with $(x,y),(x',y')\in C$. For $u=r+\dis C/2$ and $w=t+\dis C/2$, we have $u+w\geq d(x,x')$. As $X$ is $\nu$-geodesic, there exists a point $z\in X$ such that $d(x,z)\leq u+\nu$ and $d(x',z)\leq w+\nu$. We can choose a corresponding point $p\in Y$, with~$(z,p)\in C$. For this point, we get
 \[d(y,p)\leq d(x,z)+\dis C\leq (u+\nu)  +2s= (r+\dis C/2)+\nu  +2s \leq r+\nu+{3}s\]
 and similarly $d(y',p)\leq t+\nu+{3}s$. Thus $Y$ is $(\nu+3s)$-geodesic.
\end{proof}
If $X$ is an $r$-dense subset of a geodesic space $Y$, then $\dGH{X}{Y}\leq r$, and the above proposition implies that $X$ is $(3s)$-geodesic for every $s>r$. In particular, $\nu(X)\leq 3r$. Note however that in this case  \cref{subset-geodesic-small-defect} gives the stronger bound $\nu(X)\leq r$.

The persistent homology of the Vietoris--Rips filtration is commonly used in topological data analysis as a stable multi-scale signature of a finite metric space.
For length spaces, it is known that all structure maps in the first persistent homology of the Vietoris--Rips filtration are surjective \cite[Corollary 6.2]{MR3275299}. This statement generalizes to arbitrary metric spaces using the geodesic defect as follows.
\begin{proposition}
Let $X$ be a $\nu$-geodesic metric space. Then for every $2\nu < t< u$ the canonical map $H_1(\VR{t}{X})\to H_1(\VR{u}{X})$ is surjective.
\end{proposition}
\begin{proof}
Let $\{x,y\}\in \VR{u}{X}$ be an edge with length $d(x,y)=u$. As $X$ is $\nu$-geodesic, there exists a point $z\in X$ with $d(x,z)\leq \frac12 u+\nu$ and $d(y,z)\leq \frac12 u+\nu$. By assumption, we have $\frac12 u+\nu =\frac12 u+\frac12 t-(\frac12 t-\nu)<u-l$, where $l=(\frac12 t-\nu)$.  Hence, the simplex $\{z,x,y\}$ is contained in $\VR{u}{X}$ and the simplicial chain $[x,y]$ is homologous to $[z,y]-[z,x]\in C_1(\VR{u-l}{X})$. As every simplicial chain in $C_1(\VR{u}{X})$ is a finite sum of edges and $l>0$ is a constant, it follows that finitely many reapplications of the argument above yields that this chain is homologous to a chain in $C_1(\VR{t}{X})$, proving the claim.
\end{proof}
\end{arxiv}

\subsection{Rips' Contractibility Lemma via the injective hull}

In this section, we recall some known facts about embeddings of metric spaces into their injective hull.
We adapt these results using the geodesic defect to prove a version of the Contractibility Lemma for finite $\delta$-hyperbolic $\nu$-geodesic metric spaces, following \cite{LimMemoliEtAl2020}.

Let $Y$ be a metric space.
The \emph{\v{C}ech complex} of a subspace $X\subseteq Y$ for radius $r > 0$ is the nerve of the collection of closed balls in $Y$ with radius $r$ centered at points in $X$:
\[\cech_{r}(X,Y)=\{\emptyset\neq S\subseteq X\mid S\text{ finite, }\bigcap_{x\in S} \CB{r}{x}\neq\emptyset\},\]
 where $\CB{r}{x}=\{y\in Y\mid d(x,y)\leq r\}$ denotes the closed ball in $Y$ of radius $r$ centered at $x$.

 A metric space is \emph{hyperconvex} \cite{MR1904284} if it is geodesic and if any collection of closed balls has the Helly property, i.e., if any two of these balls have a nonempty intersection, then all balls have a nonempty intersection. The following lemma is a direct consequence of this definition.
\begin{lemma}
\label{rips-as-cech}
If $Y$ is hyperconvex and $X\subseteq Y$ is a subspace, then $\cech_{r}(X,Y)=\VR{2r}{X}$.
\end{lemma}

Let $X$ be a metric space. We describe its \emph{injective hull} $E(X)$, following Lang~\cite{lang}.
A function $f\colon X\to\mathbb{R}$ with $f(x)+f(y)\geq d(x,y)$ for all $x,y\in X$ is \emph{extremal} if $f(x) = \sup_{y \in X} (d(x,y)-f(y))$ for every $x\in X$.
The difference between any two extremal functions turns out to be bounded, and so we can equip the set $E(X)$ of extremal functions with the metric induced by the supremum norm, i.e., $d(f,g)=\sup_{x\in X}|f(x)-g(x)|$.
We define an isometric embedding~$e\colon X\to E(X)$ by~$y\mapsto d_y$, where~$d_y(x)=d(y,x)$.

\begin{remark}
\label{ex-properties}
$E(X)$ is a hyperconvex space.
In particular, $E(X)$ is contractible,
and nonempty intersections of closed metric balls are contractible \cite{lang,MR182949}. Moreover, nonempty intersections of open metric balls are also contractible \cite[Proposition 2.8 and Lemma 2.15]{LimMemoliEtAl2020}.
\end{remark}

The following theorem is essentially due to Lang \cite{lang}.
Originally, it has been stated for a special case, but the proof applies verbatim to the below statement involving the geodesic defect, which indeed provided the motivation for our definition. Note that the definition of $\delta$-hyperbolic used in \cite{lang} differs from the one used here by a factor of $2$.
\begin{proposition}[{Lang \cite[Proposition 1.3]{lang}}]
\label{dense-embedding-injective-hull}
Let $X$ be a $\delta$-hyperbolic $\nu$-geodesic metric space. Then the injective hull $E(X)$ is $\delta$-hyperbolic, and every point in $E(X)$ has distance at most $2\delta +\nu$ to $e(X)$.
\end{proposition}

Now we prove a generalization of the Contractibility Lemma using the injective hull analogously to the proof for geodesic spaces in \cite[Corollary 8.4]{LimMemoliEtAl2020}.
\begin{theorem}
 Let $X$ be a finite $\delta$-hyperbolic $\nu$-geodesic metric space. Then the complex~$\VR{t}{X}$ is contractible for every $t \geq 4\delta+2\nu$.
\end{theorem}
\begin{proof}
By \cref{dense-embedding-injective-hull}, we know that for $r > \frac t2 \geq 2\delta+\nu$ the collection of open balls with radius $r$ centered at the points in $e(X)$ covers $E(X)$. By finiteness of~$X$, there exists an $r>\frac t2$ such that the nerve of this cover is isomorphic to $\cech_{\frac t2}(e(X),E(X))$.
As $e$ is an isometric embedding, \cref{rips-as-cech}, \cref{ex-properties}, and the Nerve Theorem \cite[Section 4.G]{MR1867354} imply
\[
\VR{t}{X}=\VR{t}{e(X)}=\cech_{\frac t2}(e(X),E(X))\simeq E(X)\simeq *. \qedhere\]
\end{proof}

\section{Filtered collapsibility of Vietoris--Rips complexes}

In this section, we revisit the original proof of the Contractibility Lemma in \cite{Gromov1987}, adapted to the language of discrete Morse theory \cite{MR1612391}.
Focusing on the finite case, which also constitutes the key part of the original proof, we extend the statement beyond geodesic spaces using the geodesic defect, strengthen the assertion of contractibility to collapsibility, and further extend the result to become compatible with the Vietoris--Rips filtration.

\begin{theorem}
\label{discrete-rips-collapse}
Let $X$ be a finite $\delta$-hyperbolic $\nu$-geodesic metric space.
Then for every $t\geq 4\delta+2\nu$ there exists a discrete gradient that induces a collapse $\VR{t}{X}\searrow\{*\}$.
\end{theorem}
\begin{proof}
Without loss of generality, assume that $\delta>0$;
if $X$ is $0$-hyperbolic, then it is also $\epsilon$-hyperbolic for any $\epsilon>0$, and for sufficiently small $\epsilon>0$ we have $\VR{4\epsilon+2\nu}{X}=\VR{2\nu}{X}$.

Choose a reference point $p\in X$ and order the points according to their distance to $p$, choosing a total order $p=x_1<\dots<x_n$ on $X$ such that $x_i<x_j$ implies $d(x_i,p)\leq d(x_j,p)$.
Let $t\geq 4\delta+2\nu$ and consider the filtration
\[\{p\}= K_1\subseteq\cdots\subseteq K_{n}=\VR{t}{X},\]
where $K_i = \VR{t}{X_i}$ for $X_i := \{x_1,\dots,x_i\}$.
We prove that for $i\in\{2,\dots,n\}$ there exists a discrete gradient $V_i$ on $K_i$ inducing a collapse $K_{i}\searrow K_{i-1}$.

First assume %
$d(x_i,p)< t$.
Then for any vertex $x_k$ of $K_i$ we have $k \leq i$ and
$d(x_k,p) \leq d(x_i,p) < t$,
so $K_i$ is a simplicial cone with apex $p$.
Pairing the simplices containing $p$ with those not containing $p$, we obtain a discrete gradient inducing a collapse $K_i\searrow K_{i-1}$:
\[V_i = \{(\sigma \setminus \{p\},\sigma\cup\{p\}) \mid \sigma \in K_i\setminus K_{i-1} \}.\]

Now assume %
$d(x_i,p)\geq t$.
We show that there exists a point $z\in X_{i-1}$ such that for every simplex $\sigma\in K_i\setminus K_{i-1}$, the union $\sigma\cup\{z\}$ is also a simplex in~$K_{i}\setminus K_{i-1}$.
To this end, we show that any vertex $y$ of $\sigma$ has distance $d(y,z) \leq t$ to $z$.
For $r=d(x_i,p)-2\delta-\nu$ and $s=2\delta+\nu$ we have $r+s=d(x_i,p)$, and therefore, by the assumption that $X$ is a $\nu$-geodesic space, there exists a point $z\in X$ with $d(z,p)\leq r+\nu = d(x_i,p)-2\delta$, implying $z< x_i$, and $d(z,x_i)\leq s+\nu = 2\delta+2\nu$.
By assumption $t \geq 4\delta+2\nu$, and thus we get $d(z,x_i) \leq t-2\delta$.
Note that $y \in X_i$ implies $d(y,p)\leq d(x_i,p)$, and $y,x_i \in \sigma$ implies $d(y,x_i)\leq \diam{\sigma}\leq t$. The four-point condition \eqref{gromov-hyperbolic-alternative} now yields
\begin{align}
\label{estimate-proof-contractibility-lemma}
d(y,z)&\leq \max\{d(y,x_i)+d(z,p),\,d(y,p)+d(z,x_i) \}+2\delta-d(x_i,p)\nonumber\\
&= \max\{\underbrace{d(y,x_i)}_{\leq t}+\underbrace{d(z,p)-d(x_i,p)}_{\leq-2\delta},\,\underbrace{d(y,p)-d(x_i,p)}_{\leq 0}+\underbrace{d(z,x_i)}_{\leq t-2\delta} \} +2\delta\leq t.
\end{align}
Similarly to the above, pairing the simplices containing $z$ with those not containing $z$ yields a discrete gradient inducing a collapse $K_i\searrow K_{i-1}$:
\[V_i = \{(\sigma \setminus \{z\},\sigma\cup\{z\}) \mid \sigma \in K_i\setminus K_{i-1} \}.\]

Finally, by \cref{union-of-gradients}, the union $V=\bigcup_i V_i$ is a discrete gradient on $\VR{t}{X}$ and by \cref{collapsing-theorem} it induces a collapse~$\VR{t}{X}\searrow \{p\}$.
\end{proof}
 \begin{remark}
 \label{remark-strong-collapse}
For a simplicial complex $K$, a particular type of simplicial collapse called an \emph{elementary strong collapse} from $K$ to $K\setminus \St{v}$ is defined in \cite{barmak-minian} for the case where the link of the vertex~$v$ is a simplicial cone. The proof of \cref{discrete-rips-collapse} actually shows that for~$t\geq 4\delta+2\nu$ there exists a sequence of elementary strong collapses from $\VR{t}{X}$ to $\{*\}$.
 \end{remark}

We can now extend the proof strategy of \cref{discrete-rips-collapse} to obtain a filtration-compatible strengthening of the Contractibility Lemma.
\begin{proof}[Proof of \cref{discrete-rips-fitlered-collapse}]
As in the proof of \cref{discrete-rips-collapse}, we can assume that $\delta>0$, and order the points in $X$ according to their distance to a chosen reference point $p=x_1<\dots<x_n$.

As $X$ is finite, we can enumerate the values of pairwise distances by $0=r_0<\dots<r_l$.
For every $r_m>4\delta+2\nu$ we construct a discrete gradient $W_m$ inducing a collapse $\VR{r_m}{X}\searrow\VR{r_{m-1}}{X}$.
This will prove the theorem, because it follows from \cref{discrete-rips-collapse} that there exists a discrete gradient $V$ that induces a collapse $\VR{4\delta+2\nu}{X}\searrow\{*\}$, and an application of \cref{union-of-gradients} assembles these gradients into a single gradient $W=V\cup \bigcup_m W_m$on $\cl{(X)}$ inducing collapses
 $\VR{u}{X}\searrow \VR{t}{X} \searrow \{*\}$ for every~$u> t\geq 4\delta+2\nu$.

Let $m$ be arbitrary such that $r_m>4\delta+2\nu$. Consider the filtration
\[\VR{r_{m-1}}{X}=K_1\subseteq\cdots\subseteq K_{n}=\VR{r_{m}}{X},\]
where $K_i = \VR{r_{m-1}}{X} \cup \VR{r_{m}}{X_i}$ for $X_i := \{x_1,\dots,x_i\}$.
We prove that for $i\in\{2,\dots,n\}$ there exists a discrete gradient~$V_i$ on $K_i$ inducing a collapse $K_i\searrow K_{i-1}$. Note that $K_i\setminus K_{i-1}$ consists of all simplices of diameter $r_m$ that contain $x_i$ as the maximal vertex.

First assume %
$d(x_i,p)< r_m$.
Let $\sigma\in K_{i} \setminus K_{i-1}$.
As $x_i$ is the maximal vertex of $\sigma$, we have $d(v,p)\leq d(x_i,p)< r_m$ for all $v\in\sigma$.
Since $\sigma$ has diameter~$%
r_m$%
, this implies that $\sigma\cup\{p\}$ also has diameter $r_m$%
. Moreover, this implies that there exists an edge~$e\subseteq \sigma\setminus \{p\}\subseteq \sigma$ not containing $p$ with $\diam e=r_m$. Therefore, $\sigma\setminus \{p\}$ also has diameter~$r_m$. As $p<x_i$, both simplices $\sigma\setminus \{p\}$ and $\sigma\cup \{p\}$ contain $x_i$ as the maximal vertex and are thus contained in
$K_{i} \setminus K_{i-1}$.
Pairing the simplices containing $p$ with those not containing $p$, we obtain a discrete gradient inducing a collapse $K_i\searrow K_{i-1}$:
\[V_i = \{(\sigma \setminus \{p\},\sigma\cup\{p\}) \mid \sigma \in K_i\setminus K_{i-1} \}.\]
 
Now assume %
$d(x_i,p)\geq r_m$.
We show that there exists a point $z\in X_{i-1}$ such that for every simplex $\sigma\in K_{i}\setminus K_{i-1}$, the simplices $\sigma\setminus\{z\}$ and $\sigma\cup\{z\}$ are also contained in $K_{i}\setminus K_{i-1}$.
To this end, we show first that any vertex $y$ of $\sigma$ has distance~$d(y,z) \leq r_m$ to~$z$.
As in the proof of \cref{discrete-rips-collapse}, there exists a point $z\in X$ with $d(z,p)\leq d(x_i,p)-2\delta$, implying $z < x_i$, and $d(z,x_i)\leq 2\delta+2\nu$.
By assumption $r_m>4\delta+2\nu$, and thus we get $d(z,x_i)< r_m-2\delta$.
Similar to \cref{estimate-proof-contractibility-lemma}, we have the following estimate
 \begin{align*}
 d(y,z)&\leq \max\{\underbrace{d(y,x_i)}_{\leq r_m}+\underbrace{d(z,p)-d(x_i,p)}_{\leq-2\delta},\, \underbrace{d(y,p)-d(x_i,p)}_{\leq 0}+\underbrace{d(z,x_i)}_{< r_m-2\delta} \} +2\delta\leq r_m ,
 \end{align*}
and if $d(y,x_i)<r_m$, then $d(y,z)<r_m$.
Hence, $\diam{(\sigma\cup\{z\})}=r_m$, and $\diam{\sigma}=r_m$
implies
$\diam\sigma\setminus \{z\}=r_m$, by an argument similar to the above.
As $z<x_i$, both simplices $\sigma\setminus \{z\}$ and $\sigma\cup \{z\}$ contain~$x_i$ as the maximal vertex and are thus contained in
$K_{i} \setminus K_{i-1}$.
Pairing the simplices containing~$z$ with those not containing $z$, we obtain a discrete gradient inducing a collapse~$K_i\searrow K_{i-1}$:
\[V_i = \{(\sigma \setminus \{z\},\sigma\cup\{z\}) \mid \sigma \in K_i\setminus K_{i-1} \}.\]

By \cref{union-of-gradients} the union $W_m=\bigcup V_i$ is a discrete gradient on $\VR{r_m}{X}$, and by \cref{collapsing-theorem} it induces a collapse~$\VR{r_m}{X}\searrow \VR{r_{m-1}}{X}$. 
\end{proof}

\begin{arxiv}
\end{arxiv}
\begin{arxiv}
 \begin{remark}
 We do not know whether this bound is tight for spaces with positive hyperbolicity. However, it is sharp for every finite tree metric, as can be deduced from \cref{geod-defect-tree}.
 \end{remark}
\end{arxiv}

\section{Collapsing Vietoris--Rips complexes of trees by apparent pairs}

In this section, we analyze the Vietoris--Rips filtration of a tree metric space $(V,d)$ for a positively weighted finite tree $T=(V,E)$, with the goal of proving the collapses in \cref{discrete-rips-tree-fitlered-collapse} using the apparent pairs gradient. To this end, we introduce two other discrete gradients: the \emph{canonical gradient}, which is independent of any choices, and the \emph{perturbed gradient}, which coarsens the canonical gradient and can be interpreted as a gradient that arises through a symbolic perturbation of the edge lengths.
We then show that the intervals in the perturbed gradient are refined by apparent pairs of the lexicographically refined Vietoris--Rips filtration, with respect to a particular total order on the vertices.

We write $\CB{r}{x}=\{y\in V\mid d(x,y)\leq r\}$ and $\Sph{r}{x}=\{y\in V\mid d(x,y) = r\}$.
\begin{lemma}
\label{max-simp-on-sphere}
  Let $x,y\in V$ be two distinct points at distance $d(x,y)=r$.
  Then we have~$\diam{\CB{r}{x}\cap\CB{r}{y}}=r$.
  Furthermore, if $a,b\in \CB{r}{x}\cap\CB{r}{y}$ are points with $d(a,b)=r$, then these points are contained in the union $\Sph{r}{x}\cup\Sph{r}{y}$.
\end{lemma}
\begin{proof}
We start by showing the first claim. Let $a,b\in \CB{r}{x}\cap\CB{r}{y}$ be any two points. We show that $d(a,b)\leq  r$ holds, implying $\diam{\CB{r}{x}\cap\CB{r}{y}}\leq r$. Because $x,y\in \CB{r}{x}\cap\CB{r}{y}$ we also have $ \diam{\CB{r}{x}\cap\CB{r}{y}}\geq r$, proving equality.

 Write $[n]=\{1,\dots,n\}$ and let $\gamma\colon ([n],\{\{i,i+1\}\mid i\in [n-1]\})\to T$ be the unique shortest path $x\rightsquigarrow y$. Moreover, let $\Psi_a$ and $\Psi_b$ be the unique shortest paths $x\rightsquigarrow a$ and $x\rightsquigarrow b$, respectively. Consider the largest numbers $t_a,t_b\in [n]$ with $\gamma(t_a)=\Psi_a(t_a)$ and $\gamma(t_b)=\Psi_b(t_b)$ and assume without loss of generality $t_a\leq t_b$. Note that the unique shortest path $a\rightsquigarrow b$ is then given by the concatenation $a\rightsquigarrow \gamma(t_a)\rightsquigarrow\gamma(t_b)\rightsquigarrow b$, where $\gamma(t_a)\rightsquigarrow\gamma(t_b)$ is the restricted path $\gamma_{\mid [t_a,t_b]}$.
 By assumption, we have $d(a,y)\leq r$ and this implies the inequality
 \begin{align*}
  d(a,\gamma(t_a))+d(\gamma(t_a),y)=d(a,y)\leq r =d(x,y) =  d(x,\gamma(t_a))+d(\gamma(t_a),y),
 \end{align*}
 which is equivalent to $d(a,\gamma(t_a))\leq d(x,\gamma(t_a))$. Similarly, the assumption $d(x,b)\leq r$ implies $d(\gamma(t_b),b)\leq d(\gamma(t_b),y)$. Thus, the distance $d(a,b)$ satisfies \begin{align}
\label{inequ_path}
d(a,b)&=d(a,\gamma(t_a))+d(\gamma(t_a),\gamma(t_b))+d(\gamma(t_b),b) \nonumber \\                                                                                                                        &\leq d(x,\gamma(t_a))+d(\gamma(t_a),\gamma(t_b))+d(\gamma(t_b),y)=d(x,y)=r,                                                                                                                            \end{align}
which finishes the proof of the first claim.

We now show the second claim; assume $d(a,b)=r$. From the inequalities \eqref{inequ_path} and $d(a,\gamma(t_a))\leq d(x,\gamma(t_a))$, $d(\gamma(t_b),b)\leq d(\gamma(t_b),y)$ together with the assumption $d(a,b)=r$, we deduce the equalities $d(a,\gamma(t_a))= d(x,\gamma(t_a))$ and $d(\gamma(t_b),b)= d(\gamma(t_b),y)$. Hence, \[d(a,y)=d(a,\gamma(t_a))+d(\gamma(t_a),y)=d(x,\gamma(t_a))+d(\gamma(t_a),y)=d(x,y)=r \]
    and similarly $d(x,b)=r$, proving the second claim.
  \end{proof}

Enumerate the values of pairwise distances by $0=r_0<\dots < r_l = \diam V$.
Let $\K{m}:= \VR{r_{m-1}}{V} \cup T_{r_m}$.
We show that the complement~$\RC{m} :=\VR{r_m}{V}\setminus \K{m}$ is the set of all cofaces of non-tree edges of length $r_m$.
We further show that it is partitioned into regular intervals in the face poset, and that this constitutes a discrete gradient.

  \begin{lemma}
  \label{contained-in-max-simplex}
  Every edge $e\in \VR{r_{m}}{V}\setminus \VR{r_{m-1}}{V}$ is contained in a unique maximal simplex $\Delta_e\in \VR{r_{m}}{V}\setminus \VR{r_{m-1}}{V}$. Moreover, if $e$ is a tree edge of length $r_{m}$, then $\Delta_e=e$, and if $e\in \RC{m}$, then $\Delta_e\in \RC{m}$ and $e\subsetneq \Delta_e$.
  \end{lemma}
\begin{proof}
 By definition, $e$ corresponds to two points $x,y\in V$ at distance $d(x,y)=r_{m}$. 
 If $e$ is contained in the simplex $\Delta\in \VR{r_{m}}{V}$, then the points in $\Delta$ lie in the intersection $\CB{r_{m}}{x}\cap\CB{r_{m}}{y}$, which has diameter $r_{m}$ by \cref{max-simp-on-sphere}. Hence, the maximal simplex $\Delta_e$ is spanned by all the points in $\CB{r_{m}}{x}\cap\CB{r_{m}}{y}$.
 
 If $e$ is a tree edge of length $r_m$, then this intersection only contains $x$ and $y$, and hence $\Delta_e=e$. If $e\in \RC{m}$, then this intersection contains at least one vertex different from $x$ and $y$ that lies on the unique shortest path $x\rightsquigarrow y$. This implies $e\subsetneq \Delta_e$.
\end{proof}

For every maximal simplex $\Delta\in \RC{m} \subseteq \VR{r_{m}}{V}$, we write $E_\Delta$ for the set of edges~$e\in \RC{m}$ with $\Delta_e=\Delta$.
Note that $E_\Delta$ is the set of non-tree edges of length $r_{m}$ contained in~$\Delta$.

\subsection{Generic tree metrics}
\label{non-degenerate-tree-metric}
Before dealing with the general case, let us focus on the special case where the metric space~$(V,d)$ is \emph{generic}, meaning that the pairwise distances are distinct.
In this case, \cref{contained-in-max-simplex} implies that the diameter function $\diam{}\colon \cl{(V)}\to\mathbb{R}$ is a discrete Morse function, defined on the full simplicial complex on $V$,
with discrete gradient
\[\{[e,\Delta_e]\mid \text{non-tree edge }e\subseteq \cl{(V)}\},\]
which we call the \emph{generic gradient}, and only the vertices $V$ and the tree edges $E$ are critical.
\begin{arxiv}
\begin{example}
Consider the following weighted tree with vertex set $V = \{a,b,c,d\}$:
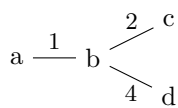
\[
\begin{tikzpicture}
 \node (A) at (0,0) {a};
 \node (B) at (1,0) {b};
 \node (C) at (2,.5) {c};
 \node (D) at (2,-.5) {d};
 \draw[-] (A) -- node[above] {\small 1} (B) ;
 \draw[-] (B) -- node[above] {\small 2} (C);
 \draw[-] (B) -- node[below] {\small 4} (D);
 \end{tikzpicture}
 \]
The generic gradient is given by $\{[\{a,c\},\{a,b,c\}],[\{a,d\},\{a,b,d\}],[\{c,d\},\{a,b,c,d\}]\}$. These intervals are the preimages under the diameter function of the non-tree distances $3$, $5$, and $6$, respectively.
 \end{example}
\end{arxiv}
Together with \cref{collapsing-theorem}, this yields the following theorem.
\begin{theorem}
\label{generic-collapse}
If the tree metric space $(V,d)$ is generic, then the generic gradient induces, for every $m\in\{1,\dots,l\}$, a sequence of collapses
\[\VR{r_{m}}{V}\searrow (\VR{r_{m-1}}{V} \cup T_{r_m}) \searrow T_{r_m} . \]
\end{theorem}

Moreover, it follows from \cref{apparent-equal-gradient} that for the Vietoris--Rips filtration, refined lexicographically with respect to an arbitrary total order on the vertices, the zero persistence apparent pairs refine the generic gradient, and therefore also induce the above collapses.

\begin{theorem}
\label{generic-apparent-pairs-collapse}
If the tree metric space $(V,d)$ is generic, then the apparent pairs gradient induces, for every $m\in\{1,\dots,l\}$, a sequence of collapses
\[\VR{r_{m}}{V}\searrow (\VR{r_{m-1}}{V} \cup T_{r_m}) \searrow T_{r_m} . \]
\end{theorem}

\subsection{Arbitrary tree metrics}
We now turn to the general case, where \cref{apparent-equal-gradient} is not directly applicable anymore, as the diameter function is not necessarily a discrete Morse function. Nevertheless, we show that \cref{generic-apparent-pairs-collapse} is still true without the genericity assumption, if the vertices $V$ are ordered in a compatible way.
Let $\Delta$ be
a maximal simplex $\Delta\in \RC{m} \subseteq \VR{r_{m}}{V}$.

  \begin{lemma}
 \label{star-is-complement}
    We have $\St{E_\Delta}=\RC{m}\cap \cl\Delta$.
 \end{lemma}
 \begin{proof}
      The inclusion $\St{E_\Delta}\subseteq \RC{m}\cap \cl\Delta$ holds by definition of $E_\Delta$.
      To show the inclusion $\St{E_\Delta}\supseteq \RC{m}\cap \cl\Delta$, let $\sigma\in \RC{m}\cap \cl\Delta$ be any simplex.
      As the Vietoris--Rips complex is a clique complex, there exists an edge~$e\subseteq \sigma\subseteq \Delta$ with $\diam{e}=r_m$. By \cref{contained-in-max-simplex}, this edge can not be a tree edge end hence $e\in \RC{m}$. Therefore, $e\in E_\Delta$ and $\sigma\in \St{e}\subseteq \St{E_\Delta}$.
 \end{proof}

\begin{lemma}
\label{maximal-simplices-independent}
 If two distinct maximal simplices $\Delta, \Delta'\in \RC{m}=\VR{r_m}{V}\setminus \K{m}$ intersect in a common face~$\Delta\cap \Delta'$, then this face is contained in $\K{m}$.
\end{lemma}
\begin{proof}
Assume for a contradiction that $\emptyset \neq \Delta\cap \Delta' \notin \K{m}$,
implying $\Delta\cap \Delta' \in \RC{m}$.
By \cref{star-is-complement}, there exists an edge $e\in E_\Delta\subseteq \RC{m}$ with $e\subseteq \Delta\cap \Delta'$, and therefore $\Delta=\Delta'$ by uniqueness of the maximal simplex containing $e$ (\cref{contained-in-max-simplex}),
a contradiction.
\end{proof}

 We denote by $L_\Delta$ the set of all vertices of~$\Delta$ that are not contained in any edge in~$E_\Delta$.
\begin{lemma}
 \label{point-on-path-interior}
Let $e=\{u,w\}\in E_\Delta$ be an edge.
Then any point $x\in V \setminus \{u,w\}$ on the unique shortest path $u\rightsquigarrow w$ of length $r_{m}$ in $T$ is contained in $L_\Delta$.
In particular, $L_\Delta$ is nonempty.
 \end{lemma}
 \begin{proof}
 By assumption, we have $d(u,x)<r_m$, $d(w,x)<r_m$ and $d(u,w)=r_m$. Therefore, $\diam{\{u,w,x\}}=r_m$ and $x\in\{u,w,x\}\subseteq\Delta_e= \Delta$.
 Assume for a contradiction that $x$ is contained in an edge in~$E_\Delta$. Then it follows from \cref{max-simp-on-sphere} that we have $d(u,x)=r_{m}$ or $d(w,x)=r_{m}$, contradicting the above.
 We conclude that $x \in L_\Delta$.
 \end{proof}

\subsubsection{The canonical gradient}
We now describe a discrete gradient that is compatible with the diameter function and induces the same collapses as in \cref{generic-collapse} even if the tree metric is not generic. This construction is \emph{canonical} in the sense that it does not depend on the choice of an order on the vertices, in contrast to the subsequent constructions.
\begin{lemma}
\label{edelta-has-good-edges}
 For any two edges $f,e\in E_\Delta$ and any vertex $v\in f$ there exists a vertex $z\in e$ such that $\{v,z\}\in E_\Delta$ is an edge in $E_\Delta$.
\end{lemma}
\begin{proof}
Let $f=\{v,w\},e=\{x,y\}$; note that $d(v,w)=d(x,y)=r_m$. Since~$f$ and $e$ are both contained in the maximal simplex $\Delta$, we have $v,w\in\CB{r_m}{x}\cap\CB{r_m}{y}$. Both $\{v,x\}$ and~$\{v,y\}$ are contained in $\{v,x,y\}\subseteq\Delta$ and \cref{max-simp-on-sphere} implies that at least one of these two edges is contained in $\VR{r_m}{V}\setminus\VR{r_{m-1}}{V}$; call this edge $e_v$. It follows from \cref{contained-in-max-simplex} that $e_v$ is not a tree edge, and therefore $e_v\in E_\Delta$.
\end{proof}

\begin{lemma}
\label{simplex-collapse}
 The set $\St{E_\Delta}=\RC{m}\cap \cl\Delta$ is partitioned by the intervals
   \begin{equation}
  \label{equ:gradient}
   W_\Delta=\{[\cup S,(\cup S)\cup L_\Delta] \mid \emptyset\neq S\subseteq E_\Delta\} ,
  \end{equation}
 and these form a discrete gradient on $\cl\Delta$ inducing a collapse $\cl\Delta\searrow (\K{m}\cap \cl{\Delta})$.
\end{lemma}
\begin{proof}
  The intervals in $W_\Delta$ are disjoint and contained in $\St{E_\Delta}$ by construction. They are regular, because $L_\Delta$ is nonempty (by \cref{point-on-path-interior}).
  By \cref{collapsing-theorem}, it remains to show that the intervals in $W_\Delta$ partition $\St{E_\Delta}=\cl\Delta\setminus (\K{m}\cap \cl{\Delta})$ and that $W_\Delta$ is a discrete gradient.
  
To show the first claim, it suffices to prove that any simplex $\sigma\in \St{E_\Delta}$ is contained in a regular interval of $W_\Delta$.
Consider the simplex $\tau=\sigma\setminus L_\Delta\subseteq \sigma$. As $\sigma\in \St{E_\Delta}$, there exists an edge $%
e\in E_\Delta$ with $e\subseteq \sigma$.
By the definition of $L_\Delta$, we have $e\subseteq \sigma\setminus L_\Delta= \tau$. Any other vertex $v\in \tau\setminus e$ is also contained in one of the edges $E_\Delta$. By \cref{edelta-has-good-edges}, there exists an edge $e_v = \{v,w\}\in E_\Delta$, where $w \in e$.
Then $\tau = e\cup\bigcup_{v\in {\tau\setminus e}}e_v$ and $\sigma \in [\tau,\tau\cup L_\Delta] \in W_\Delta$.

The second claim now follows from the observation that
the function
\[\sigma \mapsto
\begin{cases}
\dim(\sigma \cup L_\Delta) & \sigma \in \St{E_\Delta} \\
\dim\sigma & \sigma \notin \St{E_\Delta}
\end{cases}
\]
is a discrete Morse function with discrete gradient $W_\Delta$.
\end{proof}
\begin{arxiv}
\begin{example}
\label{example-generic-gradient}
Consider the following tree with vertex set $V = \{a,b,c,d\}$, whose edges all have length one:
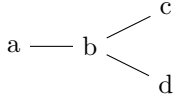
\[
\begin{tikzpicture}
 \node (A) at (0,0) {a};
 \node (B) at (1,0) {b};
 \node (C) at (2,.5) {c};
 \node (D) at (2,-.5) {d};
 \draw[-] (A) -- (B);
 \draw[-] (B) -- (C);
 \draw[-] (B) -- (D);
 \end{tikzpicture}
 \]
The complex $\VR{2}{V}$ is the full simplicial complex $\cl{V}$ with maximal simplex $\Delta=V=\{a,b,c,d\}$. We have $E_\Delta=\{\{a,c\},\{a,d\},\{c,d\}\}$ and the set $L_\Delta$ only contains the vertex~$\{b\}$. Therefore, we get \[W_\Delta=\{(\{a,c\},\{a,b,c\}),(\{a,d\},\{a,b,d\}),(\{c,d\},\{b,c,d\}),(\{a,c,d\},\{a,b,c,d\})\}.\]
 \end{example}
\end{arxiv}

Consider the union $W_m=\bigcup_\Delta W_\Delta$, where~$\Delta$ runs over all maximal simplices in $\RC{m}$ and $W_\Delta$ is as in \eqref{equ:gradient}. We call $W=\bigcup_mW_m$ the \emph{canonical gradient}.
 \begin{theorem}
 \label{tree-collapse}
  The canonical gradient is a discrete gradient on $\cl{(V)}$. For every~$m\in \{1,\dots,l\}$, it induces a sequence of collapses
  \[\VR{r_{m}}{V}\searrow  \VR{r_{m-1}}{V} \cup T_{r_{m}} \searrow T_{r_m} . \]
 \end{theorem}
 
 \begin{proof}
 Let $\Delta$ be a maximal simplex in $\Delta\in \RC{m}=\VR{r_m}{V}\setminus \K{m}$, where~$\K{m}= \VR{r_{m-1}}{V} \cup T_{r_m}$.
 It follows from \cref{simplex-collapse} that the set $W_\Delta$ is a discrete gradient on the full subcomplex~$\cl{\Delta}\subseteq \VR{r_m}{V}$ that partitions $\St{E_\Delta}=\cl{\Delta}\setminus (\K{m}\cap \cl{\Delta})$ and that induces a collapse~$\cl{\Delta}\searrow (\K{m}\cap \cl{\Delta})$.

It follows directly from \cref{maximal-simplices-independent} and \cref{union-of-gradients} that the union $W_m=\bigcup_\Delta W_\Delta$ is a discrete gradient on $\VR{r_{m}}{V}$. Again by \cref{union-of-gradients}, the union~$W=\bigcup_mW_m$ is a discrete gradient on $\cl{(V)}$.

By construction of the $W_\Delta$, the union~$W_m$ partitions the complement $\VR{r_{m}}{V}\setminus \K{m}$.
Hence, by \cref{collapsing-theorem}, it induces a collapse~$\VR{r_m}{V}\searrow \K{m} = \VR{r_{m-1}}{V} \cup T_{r_{m}}$.
Since only the vertices and the tree edges are critical for $W$, this also yields the collapse to $T_{r_m}$.
 \end{proof}

\subsubsection{The perturbed gradient}

Assume that $V$ is totally ordered. We construct a coarsening of the canonical gradient to the \emph{perturbed gradient}, such that under a specific total order of $V$ the perturbed gradient is refined by the zero persistence apparent pairs of the $\diam$-lexicographic order~$<$ on simplices.

Consider a maximal simplex $\Delta\in \RC{m}$, where $m\in \{1,\dots,l\}$. Note that all edges in $E_\Delta$ have length $r_m$ and thus are ordered lexicographically. Enumerate them as $e_1<\dots<e_q$. Every simplex $\sigma\in \RC{m}\cap\cl{\Delta}$ contains a maximal edge $e_\sigma\in \cl{\sigma}\cap E_\Delta$.

\begin{lemma}
\label{largest-edge-sigma-i}
For every edge $e_i\in E_\Delta$ the union $\Sigma_i=\bigcup_{e_\sigma=e_i}\sigma\subseteq \Delta$ is a simplex in $\RC{m}$ and the maximal edge among $\cl{\Sigma_i}\cap E_\Delta$ is $e_i$.
\end{lemma}
\begin{proof}
Note that $\Sigma_i\subseteq \Delta\in \VR{r_m}{V}$ is a simplex and it is contained in $\RC{m}$, because it is a coface of the non-tree edge $e_i$ of length $r_m$.

To prove the second claim, let $e_j\in \cl{\Sigma_i}\cap E_\Delta$ be any edge. Write $e_i=\{x,y\}$ with $x<y$ and $e_j=\{a,b\}$ with $a<b$. By construction of $\Sigma_i$, there exist simplices $\sigma_a,\sigma_b\in\RC{m}\cap\cl{\Delta}$ with $a\in\sigma_a,b\in\sigma_b$ and $e_{\sigma_a}=e_{\sigma_b}=e_i$.
Note that $\{x,y,a\}\subseteq \sigma_a$ and $\{x,y,b\}\subseteq \sigma_b$.

By \cref{max-simp-on-sphere}, we have $x,y\in \Sph{r_{m}}{a}\cup\Sph{r_{m}}{b}$ and therefore $d(a,y)=r_{m}$ (implying $a \neq y$) or $d(b,y)=r_{m}$ (implying $b \neq y$).
As $\{a,y\}\subseteq \sigma_a$ and $\{b,y\}\subseteq \sigma_b$, this implies $\{a,y\} \leq e_{\sigma_a}=e_i=\{x,y\}$ or $\{b,y\}\leq e_{\sigma_b}=e_i=\{x,y\}$, respectively. 
In particular, we have $a \leq x$ or $a<b \leq x$, and
if $a=x$, then $e_j\subseteq \sigma_b$. In any case, $e_j < e_i = e_{\sigma_b}$ as claimed.
\end{proof}

This lemma implies that $N_\Delta=\{[e_i, \Sigma_i ] \}_{i=1}^{q}$ is a collection of disjoint intervals.
It follows from \cref{simplex-collapse} that for each $j\in \{1,\dots,q\}$ the interval~$[e_j,\Sigma_j]$ is the union
\begin{align}
\label{perturbed-gradient-coarsens-canonical}
 [e_j,\Sigma_j]=\bigcup \{ [\cup S,(\cup S)\cup L_\Delta] \mid S\subseteq E_\Delta,\ e_j \text{ maximal element of }\cl{(\cup S)} \cap E_\Delta \}
\end{align}
and that $N_\Delta$
partitions $\RC{m}\cap \cl{\Delta}$. Moreover, it is the discrete gradient of the function
\begin{align}
\label{perturbed-gradient-function}
f_\Delta\colon \cl{\Delta}\to \mathbb{R},\ \sigma\mapsto\begin{cases}
                                                           i & \sigma\in [e_i, \Sigma_i ]\\
                                                           \dim\sigma-\dim \Delta & \sigma \in \K{m}
                                                          \end{cases}
\end{align}
and the intervals are regular, because $L_\Delta$ is nonempty (\cref{point-on-path-interior}). By \cref{collapsing-theorem}, $N_\Delta$ induces a collapse $\cl{\Delta}\searrow K_m\cap\cl{\Delta}$.
Therefore, the total order on $V$ induces a symbolic perturbation scheme on the edges, establishing the situation of a generic tree metric as in \cref{non-degenerate-tree-metric}.

\begin{arxiv}
\begin{example}
\label{example-perturbed-gradient}
Recalling the tree metric from \cref{example-generic-gradient}, we get \[N_\Delta=\{[\{a,c\},\{a,b,c\}],[\{a,d\},\{a,b,d\}],[\{c,d\},\{a,b,c,d\}]\}.\] Note that this gradient is different from $W_\Delta$.
 \end{example}
\end{arxiv}

Consider the union $N_m=\bigcup_\Delta N_\Delta$, where~$\Delta$ runs over all maximal simplices in $\RC{m}$. We call $N=\bigcup_m N_m$ the \emph{perturbed gradient}. By \eqref{perturbed-gradient-coarsens-canonical}, the perturbed gradient $N$ coarsens the canonical gradient $W$. Analogously to \cref{tree-collapse}, we obtain the following result.
\begin{theorem}
\label{tree-collapse-perturbed}
 The perturbed gradient is a discrete gradient on $\cl{(V)}$.
 For every $m\in \{1,\dots,l\}$, it induces a sequence of collapses
  \[\VR{r_{m}}{V}\searrow  \VR{r_{m-1}}{V} \cup T_{r_{m}} \searrow T_{r_m} . \]
\end{theorem}
 \begin{remark}
 As the lower bounds of the intervals in the perturbed gradient are edges, it follows from \cref{tree-collapse-perturbed} that these collapses can be expressed as \emph{edge collapses} \cite{MR4117732}, a notion that is similar to the elementary strong collapses described in \cref{remark-strong-collapse}.
 \end{remark}

\subsubsection{The apparent pairs gradient}
Finally, we show that for a specific total order of $V$, which we describe next, the perturbed gradient is refined by the \emph{zero persistence apparent pairs} of the $\diam$-lexicographic order.

From now on, assume that the tree $T$ is rooted at an arbitrary vertex and orient every edge away from this point. Let $\leq_V$ be the partial order on $V$ where $u$ is smaller than $w$ if there exists an oriented path $u\rightsquigarrow w$. In particular, we have the identity path $\op{id}\colon u\rightsquigarrow u$. Note that for any two vertices $u,w\in V$ the unique shortest unoriented path $u\leftrightsquigarrow w$ can be written uniquely as a zig-zag
$u\stackrel{\gamma}{\leftsquigarrow} z \stackrel{\eta}{\rightsquigarrow} w, $
where $z$ is the greatest point with $z\leq_V u$, $z\leq_V w$, and~$\gamma,\eta$ are oriented paths in %
$T$ that intersect only in $z$. If $w\leftrightsquigarrow p$ is another unique shortest unoriented path with the zig-zag $w\stackrel{{\varphi}}{\leftsquigarrow} z' \stackrel{{\mu}}{\rightsquigarrow} p$, then we can form the following diagram
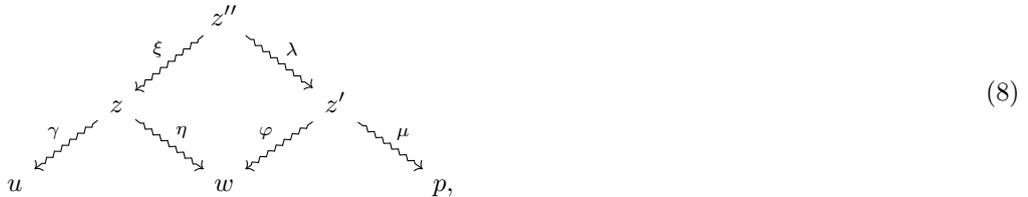
\begin{equation}
 \label{path-diagram}
\begin{tikzcd}
&& z''\arrow[swap,squiggly]{dl}{\xi}\arrow[squiggly]{dr}{\lambda} &&\\
&z\arrow[swap,squiggly]{dl}{\gamma}\arrow[squiggly]{dr}{\eta}& & z'\arrow[squiggly]{dl}[swap]{\varphi}\arrow[squiggly]{dr}{\mu}&\\                                                                                                                                                                                                                                                                                                                                                                                                                                                                                                                                               u & & w& & p,                                                                                                                                                                                                                                                                                                                                                                                                                                                                                                                                              \end{tikzcd}
\end{equation}
where $z''$ is the greatest point with $z''\leq_V z,z''\leq_V z'$. Moreover, as $T$ has no cycles, it follows that either $\xi$ or $\lambda$ is the identity path and $\varphi\circ \lambda= \eta$ or $\eta\circ \xi=\varphi$, respectively.
\begin{arxiv}
 \begin{remark}
 Note that in general the oriented paths $z''\rightsquigarrow u$ and $z''\rightsquigarrow p$ can intersect in a point different from $z''$. In particular, the zig-zag
 $u\stackrel{}{\leftsquigarrow} z'' \stackrel{}{\rightsquigarrow} p$ is not necessarily a decomposition of the unique shortest unoriented path $u\leftrightsquigarrow p$.
\end{remark}
\end{arxiv}

Extend the partial order $\leq_V$ on $V$ to a total order $<$ and consider the $\diam$-lexicographic order on simplices. As this total order on the simplices extends $<$ under the identification $v\mapsto \{v\}$, we will also denote it by $<$.
The following lemma directly implies \cref{discrete-rips-tree-fitlered-collapse}.

\begin{lemma}
\label{collapse_apparent}
The intervals in the perturbed gradient $N$ are refined by apparent pairs with respect to $<$.
For every $m\in \{1,\dots,l\}$, the zero persistence apparent pairs induce the~collapse
 \[\VR{r_{m}}{V}\searrow  \VR{r_{m-1}}{V} \cup T_{r_{m}} . \]
\end{lemma}

\begin{proof}
Consider a maximal simplex $\Delta\in \RC{m}$. Recall that $N_\Delta$ is the discrete gradient of the function $f_\Delta \colon \cl{\Delta}\to\mathbb{R}$ defined in \eqref{perturbed-gradient-function}, using the same vertex order as above.
By \cref{apparent-equal-gradient}, the zero persistence apparent pairs with respect to the $f_\Delta$-lexicographic order $<_{f_\Delta}$ are precisely the gradient pairs of the minimal vertex refinement of $N_\Delta$.

We next show that each apparent pair $(\sigma,\tau=\sigma\cup\{v\})\subseteq [e_i,\Sigma_i]$ with respect to  $<_{f_\Delta}$, where $v$ is the minimal vertex in $\Sigma_i\setminus e_i$, is an apparent pair with respect to $<$. 
Clearly, these pairs have persistence zero with respect to the diameter function, as they appear in the same interval of the perturbed gradient.
As the apparent pairs of $<_{f_\Delta}$, taken over all $\Delta$, yield a partition of $\RC{m}=\VR{r_{m}}{V}\setminus(\VR{r_{m-1}}{V} \cup T_{r_{m}})$, the same is then true for the apparent pairs of $<$. Thus, by \cref{collapsing-theorem}, the apparent pairs gradient induces a collapse $\VR{r_{m}}{V}\searrow  \VR{r_{m-1}}{V} \cup T_{r_{m}} $.

 First, let $\sigma\cup\{p\}\in \RC{m}$ be a cofacet of $\sigma$ not equal to $\tau$. We show that we must have~$\tau < \sigma\cup\{p\}$, proving that $\tau$ is the minimal cofacet of $\sigma$ with respect to $<$: If $p\in\Sigma_i$, then~$p\in \Sigma_i\setminus e_i$, as $p\notin \sigma\supseteq e_i$, and the statement is true by minimality of $v$ in the minimal vertex refinement. Now assume that~$p\notin\Sigma_i$ and write $e_i=\{u,w\}$ with $u<w$. By \eqref{perturbed-gradient-coarsens-canonical}, we have $L_\Delta\subseteq \Sigma_i$ and hence it follows that $p\notin L_\Delta$ and that the point $p$ is contained in an edge in $E_\Delta$, by definition of $L_\Delta$. It follows from \cref{max-simp-on-sphere} that $p$ together with at least one vertex of $e_i$ forms an edge in $E_\Delta$. Call this edge~$g$; if there are two such edges, consider the larger one, and call it $g$. From $\{u,w,p\}\subseteq \Delta$ and $p\notin \Sigma_i$ we get $e_i< g$: The edge $e_i$ is not the maximal edge of the two simplex $\{u,w,p\}$, since otherwise $p$ would be contained in $\Sigma_i$. Hence, one of the two other edges is maximal, and that edge is $g$ by definition. Considering the two possible cases $g=\{u,p\}$ and $g=\{w,p\}$, we must have $u<p$. We will argue that $v<p$ holds, which proves $\tau=\sigma\cup\{v\}<\sigma\cup\{p\}$.

 Consider the diagram \eqref{path-diagram}. If $\gamma\neq \operatorname{id}$, then it follows from the fact that $e_i=\{u,w\}$ is not a tree edge that along the unique shortest path $u\leftrightsquigarrow w$ there exists a vertex $x$ distinct from~$u$ and $w$ with $x<u<p$. Then $x\in L_\Delta\subseteq \Sigma_i\setminus e_i$ by \cref{point-on-path-interior}, and as $v$ is the minimal element in $\Sigma_i\setminus e_i$, we get~$v\leq x<p$.

 If $\gamma= \operatorname{id}$, then $u=z$, and it follows from $d(w,p)\leq r_{m}$ and $p\notin e_i=\{u,w\}$ that we must have $\lambda\neq \operatorname{id}$ and $\xi=\operatorname{id}$: Otherwise $\lambda=\operatorname{id}$ and $u=z$ lies on $\varphi$. Therefore, $u$ lies on the unique shortest path from $w$ to $p$ and $d(w,p) = d(w,u) + d(u,p)=r_m+d(u,p)>r_{m}$, yielding a contradiction. Thus, the unique shortest path $(u=z)\leftrightsquigarrow p$ decomposes as $u\rightsquigarrow z'\rightsquigarrow p$, where $u\rightsquigarrow z'$ is contained in $u\rightsquigarrow z'\rightsquigarrow w$. Note that $u\neq z'$, because $\lambda\neq\operatorname{id}$. Hence, as $e_i$ is not a tree edge, the immediate successor $x$ of $u$ on the path $u\rightsquigarrow w$ is distinct from $u$ and $w$ with $x\leq z'$. This point satisfies $x\leq z'\leq p$, and it follows from \cref{point-on-path-interior} that we have~$x\in L_\Delta\subseteq \Sigma_i\setminus e_i$. Because $p\notin L_\Delta$ we even have $x<p$. Therefore, as $v$ is the minimal vertex in $\Sigma_i\setminus e_i$, it follows that $v\leq x<p$.

 It remains to prove that $\sigma$ is the maximal facet of $\tau$ with respect to $<$. We write $e_i=\{u,w\}$ with $u<w$ and $\tau=\{b_0,\dots,b_{\dim \tau}\}$ with $b_0<\dots<b_{\dim \tau}$. As $e_i\subseteq \tau$, there are indices $k_1<k_2$ with $u=b_{k_1}<b_{k_2}=w$.
 If $k_1>0$, then $v=b_0$, so $\sigma$ is of the form $\{b_1,\dots,b_{\dim \tau}\}$ and is the maximal facet of $\tau$ with respect to $<$ as claimed.
 Now assume $k_1=0$. If $\tau$ contains no edges $e\in E_\Delta$ other than $e_i$, then the facets~$\tau\setminus \{u\}$ and $\tau\setminus \{w\}$ are both contained in $\VR{r_{m-1}}{V}$, because they do not contain any edge of length $r_m$, and the maximal facet of $\tau$ is $\tau\setminus \{x\}$ with $x$ the minimal vertex in $\tau\setminus e_i$. By assumption, we have $x=v$ and hence $\tau\setminus \{x\}=\tau\setminus \{v\}=\sigma$. If $\tau$ contains other edges~$e\neq e_i$ with $e\in E_\Delta$, label them $s_1,\dots,s_a$. As $e_i\subseteq \tau\subseteq \Sigma_i$, it follows from \cref{largest-edge-sigma-i} that we have $s_b<e_i$ for all $b$. Because of this and our assumption $k_1=0$, i.e., $u$ is the minimal vertex of $\tau$, we have~$s_b=\{u,x_b\}<\{u,w\}=e_i$ with $u<x_b<w$. Therefore, the facet~$\{b_1,\dots,b_{\dim\tau}\}$ contains no edges in $E_\Delta$ and hence it is contained in $\VR{r_{m-1}}{V}$. The facet $\{b_0,b_2,\dots,b_{\dim\tau}\}$ of $\tau$ contains $e_i$, hence it is an element of $\RC{m}$, and so it is maximal among the facets containing $b_0$, implying that it is the maximal facet of $\tau$ with respect to~$<$. Because $b_1$ is the minimal vertex in $\tau\setminus e_i$ and $v\in \tau\setminus e_i$, it follows from the minimality of $v\in\Sigma_i\setminus e_i$ that we have $b_1=v$, implying $\{b_0,b_2,\dots,b_{\dim\tau}\}=\sigma$. Therefore, $\sigma$ is the maximal facet of $\tau$ with respect to $<$.
\end{proof}

  \begin{remark}
   The preceding \cref{collapse_apparent} also implies \cref{discrete-rips-fitlered-collapse} in the special case of tree metrics: if~$u> t\geq 2\nu(V)=\max_{e\in E}{l}(e)$ are real numbers, then $T_t=T$ is the entire tree, and we obtain collapses~$\VR{u}{V}\searrow \VR{t}{V}\searrow T\searrow \{*\}$.
   If all edges of $T$ have the same length, it turns out that the collapse $T\searrow \{*\}$ is also induced by the apparent pairs gradient for the same order $<$.
  \end{remark}
  \begin{arxiv}
\begin{remark}
 For metrics other than tree metrics, 
 the collapse $\VR{u}{X}\searrow \VR{t}{X}$ 
 from \cref{discrete-rips-fitlered-collapse}
  is not always achieved by the apparent pairs gradient.
 Consider the following weighted graph:
  
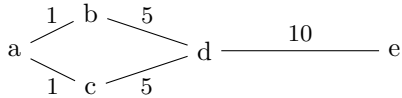
\[
\begin{tikzpicture}
 \node (a) at (0,0) {a};
 \node (b) at (1,.5) {b};
 \node (c) at (1,-.5) {c};
 \node (d) at (2.5,0) {d};
 \node (e) at (5,0) {e};
 \draw[-,midway] (a) -- node[above] {\small 1} (b);
 \draw[-,midway] (a) -- node[below] {\small 1} (c);
 \draw[-,midway] (b) -- node[above] {\small 5} (d);
 \draw[-,midway] (c) -- node[below] {\small 5} (d);
 \draw[-,midway] (d) -- node[above] {\small 10} (e);

 \end{tikzpicture}
 \]
 For this graph metric, the hyperbolicity is $\delta(X)=1$ and the geodesic defect is $\nu(X)=5$. Therefore, we have $4\delta(X)+2\nu(X)=14$. The maximal Vietoris--Rips complex has $31$ simplices in total. For the apparent pairs gradient only the simplices $\{b,e\}$ and $\{b,d,e\}$ are critical, and both have diameter $15$. Thus, the collapse $\VR{15}{X}\searrow\VR{14}{X}$ is not induced by the apparent pairs gradient.
\end{remark}
\end{arxiv}
\begin{arxiv}
\begin{example}
Revisiting the tree metric from \cref{example-generic-gradient} once more, we see that the apparent pairs of the lexicographically refined Vietoris--Rips filtration with diameter two are
\[(\{a,c\},\{a,b,c\}),(\{a,d\},\{a,b,d\}),(\{c,d\},\{a,c,d\}),(\{b,c,d\},\{a,b,c,d\}).\]
Note that together with \cref{example-perturbed-gradient} this shows that the canonical gradient, the perturbed gradient, and the apparent pairs gradient can all be different in general.
 \end{example}
\end{arxiv}

\bibliography{literature_discrete_rips_lemma.bib}

\end{document}